\documentclass[a4paper,reqno]{amsart}
\usepackage{amssymb} % 数学記号
\usepackage{amsmath} % 数学記号
\usepackage{amsthm} % 定理環境
\usepackage{ascmac} % 囲み系の環境
\usepackage{bm} % 太文字
\usepackage{color} % 文字色の追加
\usepackage[OT2,T1]{fontenc} % キリル文字用
\usepackage{fullpage} % レイアウト
\usepackage{indentfirst} % 段落変更時に字下げをする
\usepackage{mathrsfs} % 花文字
\usepackage{mathtools} % showonlyrefs等の便利ツール
\usepackage{multirow} %
\usepackage{stmaryrd} % \longmapsfrom
\usepackage{tikz} % 図の作成
\usetikzlibrary{cd}
\allowdisplaybreaks[4]
\mathtoolsset{showonlyrefs}
\theoremstyle{plain}
\newtheorem{thm}{Theorem} % 定理等の最初の数字もequationと同じものにする.
\newtheorem{prop}[thm]{Proposition}
\newtheorem{lem}[thm]{Lemma}

\numberwithin{thm}{section} % セクション番号から定理環境の番号付けをする.
\theoremstyle{definition}

\newtheorem{rem}[thm]{Remark}
\makeatletter
\newcommand{\MyAlphabet}[1]{\@tfor\Ch@r:=#1\do{
	\expandafter\edef\csname bb\Ch@r\endcsname{\noexpand\mathbb{\Ch@r}}
	\expandafter\edef\csname cal\Ch@r\endcsname{\noexpand\mathcal{\Ch@r}}
	\expandafter\edef\csname frak\Ch@r\endcsname{\noexpand\mathfrak{\Ch@r}}
	\expandafter\edef\csname scr\Ch@r\endcsname{\noexpand\mathscr{\Ch@r}}
}}
\makeatother
\MyAlphabet{ABCDEFGHIJKLMNOPQRSTUVWXYZabcdefghijklmnopqrstuvwxyz}

\let\epsilon\relax
\DeclareMathOperator{\epsilon}{\varepsilon}
\makeatletter
\newcommand{\MyMathOperators}[1]{\@for\op:=#1\do{
	\expandafter\edef\csname\op\endcsname{\noexpand\mathop{\noexpand\mathrm{\op}}\nolimits}
}}
\makeatother
\MyMathOperators{%
	Ker,%
	Coker,%
	Hom,%
	End,%
	Aut,%
	id,%
	Gal,%
	ch,%
	ord,%
	Spec,%
	Specm,%
	Frac,%
	rank,%
	corank,%
	sgn,%
	alg,%
	Reg,%
	Sel,%
	Cl,%
	lcm,%
	gcd,%
	Ann,%
	alg,%
}
\renewcommand{\bar}[1]{\overline{#1}}

\newcommand{\red}[1]{\textcolor{red}{#1}} % コメント用
\DeclarePairedDelimiter{\skakko}{\lparen}{\rparen} %括弧( )
\DeclarePairedDelimiter{\mkakko}{\lbrace}{\rbrace} %括弧{ }
\DeclarePairedDelimiter{\lkakko}{\lbrack}{\rbrack} %括弧[ ]
\DeclarePairedDelimiter{\gkakko}{\langle}{\rangle} %括弧< >
\DeclarePairedDelimiterX{\set}[2]{\lbrace}{\rbrace}{#1\,\delimsize\vert\,#2} %集合{ }
\DeclareSymbolFont{cyrletters}{OT2}{wncyr}{m}{n}
\DeclareMathSymbol{\Sha}{\mathalpha}{cyrletters}{"58}
\usepackage[dvipdfmx]{hyperref}
\hypersetup{colorlinks=true,urlcolor=blue,citecolor=blue,linkcolor=blue}
\title{The 2-adic valuations of the algebraic central $L$-values for quadratic twists of weight 2 newforms}
\author{Taiga Adachi}
\address{Joint Graduate School of Mathematics for Innovation, Kyushu University, Motooka 744, Nishi-ku Fukuoka 819-0395, Japan}
\email{t.adachi1729@gmail.com}
\author{Keiichiro Nomoto}
\address{Koden Electronics Co., Ltd.}
\email{k-nomoto@koden-electronics.co.jp}
\author{Ryota Shii}
\address{Graduate School for Mathematics, Kyushu University, Motooka 744, Nishi-ku Fukuoka 819-0395, Japan}
\email{shii.ryota@gmail.com}
\date{}
\begin{document}

\begin{abstract}
    Let $f$ be a normalized newform of weight 2 on $\Gamma_0(N)$ whose coefficients lie in $\bbQ$ and let $\chi_M$ be a primitive quadratic Dirichlet character with conductor $M$. In this paper, under mild assumptions on $M$, we give a sharp lower bound for the 2-adic valuation of the algebraic part of the $L$-value $L(f, \chi_M, 1)$ and evaluate the 2-adic valuation for infinitely many $M$.
\end{abstract}

\maketitle

\section{Introduction}\label{sec:Intro}

\noindent
The aim of this paper is to give a sharp lower bound for the 2-adic valuations of the algebraic central $L$-values for a certain family of quadratic twists of an elliptic curve defined over $\bbQ$. The 2-adic valuation of the algebraic central $L$-value for an elliptic curve is closely related to the order of the 2-primary part of the Tate--Shafarevich group if the analytic rank of the elliptic curve is zero. Also, according to Goldfeld conjecture, the set of square-free integers $m$ for which the quadratic twist of an elliptic curve by the extension $\bbQ(\sqrt{m})/\bbQ$ has analytic rank $0$ is expected to have density $1/2$. Therefore, we also aim to evaluate the 2-adic valuations for infinitely many quadratic twists.

In \cite{zhao97} and \cite{zhao01}, Zhao gave a lower bound for the $2$-adic valuations of the algebraic parts for a family of CM elliptic curves of the form $y^2=x^3-D^2x$ under some assumptions on $D$. His method using induction on the number of prime divisors of $D$ is highly versatile and currently has been applied to various elliptic curves. For example in \cite{Nomoto:2023}, the second author of this paper calculated the 2-adic valuations for a wider family of elliptic curves than the one discussed by Zhao. Kezuka \cite{Kezuka:2018} applied Zhao's method to a certain family of elliptic curves having complex multiplication by $\bbZ[(-1+\sqrt{-3})/2]$. For elliptic curves having complex multiplication by a ring which is neither $\bbZ[\sqrt{-1}]$ nor $\bbZ[(-1+\sqrt{-3})/2]$, there are Coates--Li's result \cite{CL:2020} and Coates--Kim--Liang--Zhao's result \cite{cltz}.

There are also several works, using Zhao's method, on elliptic curves defined over $\bbQ$ that do not necessarily have complex multiplication. In \cite{Zhai:2016}, Zhai evaluated the 2-adic valuations for a family of quadratic twists $E^{(m)}$ of elliptic curves when $m$ satisfies $m\equiv 1\bmod 4$ and additional assumptions. He gave the lower bound for the 2-adic valuations for a similar family of quadratic twists in \cite{Zhai:2020-2}. Cai--Li--Zhai \cite{CLZ:2020} determined the 2-adic valuations for a certain family of quadratic twists $E^{(m)}$ with $m\equiv 1\bmod 4$. In \cite{Zhai:2020-1}, Zhai also gave a stronger lower bound for the 2-adic valuations for a subfamily of \cite{Zhai:2020-2} and determined the 2-adic valuation in some cases by combining the result of \cite{CLZ:2020}. Moreover, Zhai gave lower bounds for the 2-adic valuations in terms of local Tamagawa factors in \cite{Zhai:2021}. All results mentioned here calculated only for quadratic twists $E^{(m)}$ where $m\equiv 1\bmod 4$. In this paper, we give lower bounds for 2-adic valuations in general setting,  including the case $m\equiv 3\bmod 4$ and most of the cases discussed in \cite{Zhai:2016}, \cite{Zhai:2020-2}, \cite{CLZ:2020} and \cite{Zhai:2020-1}.
\\

Before stating theorems, we introduce several notations. (See \S \ref{sec:Preliminaries} for details.) Let $E$ be an elliptic curve defined over $\bbQ$. By the modularity theorem, the $\bbQ$-isogeny class of $E$ corresponds to a normalized Hecke eigenform $f\in S_2(\Gamma_0(N))^{\mathrm{new}}$, namely $L(E/\bbQ, s)=L(f, s)$. Therefore, in the rest of the paper, we write almost everything in terms of $f$. Let $m\geq 1$ be a square-free odd integer and $\epsilon\in\{\pm 1\}$. Then, the $L$-function of the quadratic twist $E^{(\epsilon \! m)}$ is equal to the twisted $L$-function $L(f, \chi_M, s)$. Here, $\chi_M$ is a primitive quadratic character with conductor $M$, where $M$ is given by
\begin{align}
    M=\begin{cases}
        m & (\epsilon m\equiv 1\bmod 4),\\
        4m & (\epsilon m\equiv 3\bmod 4).
    \end{cases}
\end{align}
Denote the period lattice associated to $f$ by $\calL_f$. Then, there exist periods $\Omega_{f}^{\pm}\in\bbR$ such that
\begin{align}
    \calL_{f}=\Omega_{f}^{+}\bbZ+i\Omega_{f}^{-}\bbZ, \quad \text{or} \quad
    \calL_{f}=\Omega_{f}^{+}\bbZ+\dfrac{\Omega_{f}^{+}+i\Omega_{f}^{-}}{2}\bbZ.
\end{align}
The period lattice $\calL_f$ is called \textit{rectangular} in the former case and \textit{non-rectangular} in the latter case. 
We define the set of prime numbers as
\begin{align}
    \calS_{i}\coloneqq\set*{q: \text{odd prime}}{q\nmid N, v_2(a_{q}-2)=i},
\end{align}
where $a_{q}$ is the $q$-th Fourier coefficient of $f$ and $v_2$ is the 2-adic valuation of $\bar{\bbQ}$ normalized so that $v_2(2)=1$. The main result is the following.

\begin{thm}\label{thm:MainTheorem}
    Assume that the Manin constant for $E$ is equal to 1. (See \S \ref{sec:Preliminaries}.) Write $M=4^nm=4^nq_1\cdots q_r$, where $n\in\{0, 1\}$ and $q_1, \dots, q_r$ are distinct odd primes. We set $\frakv_m\coloneqq\min_{1\leq i\leq r}\mkakko*{v_2(a_{q_i}-2)}$. We put $\frakw_m$ as $0$ if $\frakv_m=0$ and $1$ otherwise. If $q_i\in \calS_{0}\cup\calS_{1}\cup\calS_{2}$ for all $i$, there exists an integer $c$ independent of $r$ such that 
    \begin{align}
        v_2\skakko*{\frac{L(f, \chi_{M}, 1)}{\Omega_{f}^{\sgn(\chi_{M})}}}
        \geq \frakw_m\cdot r+c, \label{eq:MainTheorem}
    \end{align}
   where $\sgn(\chi_M)$ is the sign of $\chi_M(-1)$.
    Moreover, there exist infinitely many $M$ for which the equality in the inequality \eqref{eq:MainTheorem} holds. Hence, the Mordell--Weil group $E^{(\epsilon \! m)}(\bbQ)$ and the Tate--Shafarevich group $\Sha(E^{(\epsilon \! m)}/\bbQ)$ are finite for such $M$.
\end{thm}

\begin{rem}
    In Theorem \ref{thm:MainTheorem}, the constant $c$ is later given explicitly in terms of $v_2(L(f, 1)/\Omega_f^+)$ and $v_2(L(f, \chi_4, 1)/\Omega_f^-)$.
    Moreover, some of the conditions that the equality holds can be written explicitly in terms of the 2-adic valuations and $q_i \bmod 4$. See \S \ref{sec:Preliminaries} for the definition of $\Omega_f^{\pm}$ and Theorem \ref{thm:MainThm1}--\ref{thm:MainThm4} for the precise statements.
\end{rem}

% \begin{rem}
%     In Theorem \ref{thm:MainTheorem}, the integer $c$ is given explicitly. In addition, some of the conditions that the equality holds can be written explicitly in terms of $v_2(L(f, 1)/\Omega_f^+)$, $v_2(L(f, \chi_4, 1)/\Omega_f^-)$ and $q_i \bmod 4$. More precisely, see from Theorem \ref{thm:MainThm1} to Theorem \ref{thm:MainThm4}.
% \end{rem}

%\begin{rem}
%    In Theorem \ref{thm:MainTheorem}, we can obtain a similar lower bound even if we remove the condition that $q_i\in \calS_0\cup\calS_1\cup\calS_2$ for all $i$. However, in such cases, it is not guaranteed to be a sharp lower bound. See \S \ref{sec:TwoAdicVal} for more details.
%\end{rem}

\medskip
\noindent
\textbf{Relation between Theorem \ref{thm:MainTheorem} and some results.}
Let $\Omega_E^+$ be the period associated to $E$ defined in \S \ref{sec:Preliminaries}.
For a prime number $q$ with $(q, N)=1$, $N_{q}=q+1-a_q$ denotes the number of $\bbF_q$-rational points of reduction modulo $q$. We write $F$ for the field obtained by adjoining to $\bbQ$ one fixed root of the 2-division polynomial of $E$. For distinct odd primes $q_1, \dots, q_r$ and $M=m=q_1\cdots q_r$, we state relations between Theorem \ref{thm:MainTheorem} and known results. 

\begin{enumerate}
    \item In \cite{Zhai:2016}, Zhai obtained lower bounds and some equality conditions for various cases. For example in \cite[Theorem 1.1]{Zhai:2016}, he gave an equality condition when $E$ has negative discriminant, $E[2](\bbQ)=\{O\}, v_2(L(E/\bbQ, 1)/\Omega_E^+)=0$, $q_i$ is inert in $F$ for all $i$ and $M\equiv 1\bmod 4$. Note that if $E[2](\bbQ)=\{O\}$ and $q$ is inert in $F$ for a prime number $q$ with $(q, N)=1$, then $a_q$ is odd integer. This is a special case of Theorem \ref{thm:MainThm2}-(i).
    \item In \cite{Zhai:2020-2}, Zhai calculated lower bounds when $q_i\equiv 1\bmod 4$ for all $i$ and $v_2(N_{q_i})\geq 1$. Theorem \ref{thm:MainThm1}-(ii) and Theorem \ref{thm:MainThm2}-(ii) neither contain this result nor are contained in this one.
    \item In \cite{CLZ:2020}, Cai--Li--Zhai showed that the equality holds when $E[2](\bbQ)\simeq \bbZ/2\bbZ$, $v_2(L(E/\bbQ, 1)/\Omega_E^+)=-1$, $q_i\equiv 1\bmod 4$ for all $i$ and $v_2(N_{q_i})=1$. Note that if $E[2](\bbQ)\neq \{O\}$, then $a_q$ is even integer for a prime number $q$ with $(q, N)=1$. This result is a special case of Theorem \ref{thm:MainThm1}-(ii) and Theorem \ref{thm:MainThm2}-(ii).
    \item In \cite{Zhai:2020-1}, Zhai calculated lower bounds when $L(E/\bbQ, 1)\neq 0$, $E[2](\bbQ)\neq \{O\}$, $r\geq 2$, $q_i\equiv 1\bmod 4$ for all $i$ and there exists $k \ (1\leq k\leq r)$ such that $q_k\in \cup_{j\geq 2}\calS_{j}$. It neither includes Theorem  \ref{thm:MainThm1}-(ii) and Theorem \ref{thm:MainThm2}-(ii) nor is included in them.
\end{enumerate}

\medskip
\noindent
\textbf{Example of Theorem \ref{thm:MainTheorem}.} We give an example of Theorem \ref{thm:MainTheorem} in the case $M=4m$. (See \S \ref{sec:Example} for more details.) Let $f\in S_2(\Gamma_0(37))^{\mathrm{new}}$ be the eigenform
\begin{align}
    f=q-2q^2-3q^3+2q^4-2q^5+6q^6-q^7+O(q^8).
\end{align}
We see that $f$ corresponds to an isogeny class of the elliptic curve given by the equation $y^2+y=x^3-x$ with Cremona label \texttt{37a1}. In this case, the integer $c$ appearing in Theorem \ref{thm:MainTheorem} is zero. Therefore, if $q_i\in \calS_{0}\cup\calS_{1}\cup\calS_{2}$ for all $i$, then we have
\begin{align}
    v_2\skakko*{\frac{L(f, \chi_M, 1)}{\Omega_{f}^{\sgn(\chi_M)}}}\geq \frakw_m\cdot r.\label{eq:IntroExample}
\end{align}
The condition for the equality in the inequality \eqref{eq:IntroExample} holds if all primes $q_i$ dividing $m$ satisfy, $q_i\neq 37$, $q_i\equiv 1\bmod 4$ and $a_{q_i}$ is odd. For example, the following prime numbers satisfy these conditions, and in fact, there are infinitely many such prime numbers.
\begin{align}
    &41, 53, 73, 101, 149, 157, 173, 181, 197, 229, 337, 373, 397, 433, 509, 521, 593,\\
    &613, 617, 641, 673, 677, 733, 761, 773, 821, 937, 953, 1009, 1033, 1109, 1117,\\
    &1181, 1193, 1217, 1249, 1321, 1381, 1409, 1453, 1481, 1621, 1637, 1709, 1801,\\
    &1861, 1913, 1933, 1949, 1997, 2069, 2081, 2113, 2137, 2153, 2221, 2269, 2273,\\
    &2293, 2297, 2341, 2357, 2377, 2417, 2441, 2557, 2617, 2689, 2729, 2749, \cdots
\end{align}

\noindent
\textbf{Outline of the proof.} We state an outline of the proof of Theorem \ref{thm:MainTheorem} in the case $M=m$. The new approach for a calculation of 2-adic valuations is to use the modular symbol
\begin{align}
    \gkakko*{r}_{f}^{\pm}\coloneqq\pi i\int_{i\infty}^{r}f(z)dz\pm\pi i\int_{i\infty}^{-r}f(z)dz \quad  (r\in\bbP^1(\bbQ)),
\end{align}
and the quantity
\begin{align}
    \calT_{d, m}=\sum_{k\in(\bbZ/m\bbZ)^\times}\chi_{d}(k)\dfrac{1}{\Omega_{f}^{\sgn(\chi_d)}}\gkakko*{\frac{k}{d}}_{f}^{\sgn(\chi_d)},
\end{align}
where $d$ is a positive divisor of $m$. As we will discuss in the last of this introduction, for example, in \cite{Zhai:2016}, \cite{Zhai:2020-2}, \cite{CLZ:2020}, \cite{Zhai:2020-1} and \cite{Zhai:2021}, 2-adic valuations are calculated without treating $\gkakko*{r}_{f}^{+}$ and $\gkakko*{r}_{f}^{-}$ separately. Separating $\gkakko*{r}_{f}^{+}$ and $\gkakko*{r}_{f}^{-}$ allows us to deal with the case where $m\equiv 3\bmod 4$.

As in previous works, we prove the main theorem based on Zhao's method (cf. \cite{zhao97}, \cite{zhao01}). First, for a divisor $d$ of $m$, we show $\calT_{d, m}$ is equal to $L(f, \chi_d, 1)/\Omega_f^{\sgn(\chi_d)}$ up to $\bar{\bbQ}^\times$. (See Theorem \ref{thm:ShimuraManin} and Proposition \ref{prop:TDD}.) Next, we consider the decomposition
\begin{align}
    \sum_{d\mid m}\calT_{d, m}=\calT_{1, m}+\sum_{\substack{d\mid m \\ d\neq 1, m}}\calT_{d, m}+\calT_{m, m}.\label{eq:Intro-1}
\end{align}
We calculate the 2-adic valuation of the left-hand side in Proposition \ref{prop:SumBound1} and Proposition \ref{prop:SumBound2}. (In this calculation, differences between the rectangular case and the non-rectangular case appear.)
The first term on the right-hand side is equal to the algebraic part $L(f, 1)/\Omega_f^+$ up to $\bar{\bbQ}^\times$. 
(See Proposition \ref{prop:TDD}.) Since the second term on the right-hand side can be calculated from the induction hypothesis, we can consequently calculate the 2-adic valuation of the third term that is essentially equal to the algebraic part $L(f, \chi_m, 1)/\Omega_f^{\sgn(\chi_m)}$. The proof for the case $M=4m$ can be done in almost the same way.

As mentioned before, other works essentially used another modular symbol
\begin{align}
    \gkakko*{r}_{f} \coloneqq \gkakko*{r}_{f}^{+}+\gkakko*{r}_{f}^{-} \quad (r\in\bbP^1(\bbQ))
\end{align}
and the quantity
\begin{align}
    T_{d, m}\coloneqq \sum_{k\in(\bbZ/m\bbZ)^\times}\chi_d(k)\frac{1}{\Omega_f^{\sgn (\chi_{d})}}\gkakko*{\frac{k}{m}}_{f}.
\end{align}
By considering a transformation $k\mapsto -k$, we see that $\calT_{m, m}=T_{m, m}$. However, whereas each term in $\calT_{d, m}$ is algebraic, each term in $T_{d, m}$ is not necessarily algebraic. More precisely, since the periods $\Omega_f^{+}$ and $\Omega_f^{-}$ are algebraically independent for newforms corresponding to non-CM elliptic curves (cf. \cite{Schneider:1957}, \cite{Waldschmidt:2008}), the value $\gkakko*{k/m}_f/\Omega_f^{\sgn(\chi_d)}$ is not necessarily algebraic. Therefore, unlike equation \eqref{eq:Intro-1}, it is necessary to decompose $\sum_{d\mid m}T_{d, m}$ so that each term is algebraic by grouping $\{T_{d, m}\}_{d\mid m}$ based on whether the number of divisors of $d$ is even or to impose additional conditions that the sign of Dirichlet character $\chi_{q_i}$ is $+1$ for any $q_i$. On the other hand, since we work with $\gkakko*{r}_{f}^{+}$ and $\gkakko*{r}_{f}^{-}$ separately, we can calculate each term $\calT_{d, m}$ without grouping them and imposing  the  conditions on the signs of Dirichlet characters.
\\

\noindent
\textbf{Organization of the paper.}
In \S \ref{sec:Preliminaries}, we give definitions of periods and modular symbols.
In \S \ref{sec:TwistLModSymb}, we discuss various properties of the quantity $\calT_{D, M}$.
In \S \ref{sec:TwoAdicVal}, we state the exact statement of Theorem \ref{thm:MainTheorem} in four theorems and prove them.
In \S \ref{sec:Example}, we give numerical examples of Theorem \ref{thm:MainTheorem}.
\\

\noindent
\textbf{Acknowledgements.}
The authors would like to express their sincere gratitude to Shinichi Kobayashi for introducing them to this topic and encouraging them. They are very grateful to Shuai Zhai, Takumi Yoshida and the anonymous referee for their valuable comments on an earlier version. The authors would also like to thank Yu Katagiri for helpful discussions. The first author was supported by WISE program (MEXT) at Kyushu University. The second and third authors were supported by JST SPRING, Grant Number JPMJSP2136.

\section{Preliminaries}\label{sec:Preliminaries}

\noindent
Let $f \in S_{2}(\Gamma_{0}(N))^{\mathrm{new}}$ be a normalized Hecke eigenform.
Assume that $f$ is defined over $\bbQ$, that is, all coefficients of the $q$-expansion of $f$ lie in $\bbQ$. Then, $f$ corresponds to a $\bbQ$-isogeny class of elliptic curves defined over $\bbQ$. We write the optimal elliptic curve in the isogenous class as $E$. There exists a non-constant rational map $\Phi:X_{0}(N) \to E$ defined over $\bbQ$. Fix a global minimal Weierstrass model of $E$ and let $\omega$ be the N\'eron differential form on $E$. Since $f(q)dq/q$ is a basis for the $\bbQ$-vector space $S_{2}(\Gamma_{0}(N))$, there exists $\nu_{E} \in \bbQ^{\times}$ such that 
\begin{align}
    \nu_{E}f(q)\frac{dq}{q} = \Phi^{*}\omega_{E}. \label{eq:manin}
\end{align}
The number $\nu_{E}$ is called the \textit{Manin constant}. It is well-known that $\nu_{E} \in \bbZ$ and conjectured that $\nu_{E} = 1$. In the rest of the paper, \textit{we assume that $\nu_{E}=1$}.  We define the period lattice $\calL_{f}$ for $f$ as
\begin{align}
    \calL_{f}:=\left\{ \int_{\gamma} f(q)\frac{dq}{q} ~\vline~ \gamma \in H_{1}(X_{0}(N)(\bbC), \bbZ) \right\} \subset \bbC.
\end{align}
Since there exists a surjective map $\Gamma_{0}(N) \longrightarrow H_{1}(X_{0}(N), \bbZ)$ (for example, see \cite[Theorem 12.1]{manin1}), 
we can take some $r \in \bbP^{1}(\bbQ)$ and have
\begin{align}
    \calL_{f}=\left\{ 2\pi i \int_{r}^{g \cdot r} f(z)dz ~\vline~ g \in \Gamma_{0}(N) \right\}.
\end{align}
Similarly, one can define the period lattice $\calL_{E}$ of $\omega_{E}$.
It is known that $\calL_{E}$ is of the form
\begin{align}
    \calL_{E} = \begin{cases}
        \Omega_{E}^{+}\bbZ+i\Omega_{E}^{-}\bbZ & (\Delta(E)<0), \vspace{5pt} \\
        \Omega_{E}^{+}\bbZ+\dfrac{\Omega_{E}^{+}+i\Omega_{E}^{-}}{2} \bbZ & (\Delta(E)>0),
    \end{cases}\label{eq:LatticeE}
\end{align}
where $\Delta(E)$ is the discriminant of the fixed model of $E$. By the equation \eqref{eq:manin} and the assumption $\nu_{E}=1$, we see that $\calL_{E}=\calL_{f}$. Therefore we have
\begin{align}
    \calL_{f} = \begin{cases}
        \Omega_{f}^{+}\bbZ+i\Omega_{f}^{-}\bbZ & (\Delta(E)<0), \vspace{5pt} \\
        \Omega_{f}^{+}\bbZ+\dfrac{\Omega_{f}^{+}+i\Omega_{f}^{-}}{2} \bbZ & (\Delta(E)>0),
    \end{cases}\label{eq:Latticef}
\end{align}
where $\Omega_{f}^{\pm}=\Omega_{E}^{\pm}$. For $r \in \bbP^{1}(\bbQ)$, set
\begin{align}
    \gkakko*{r}^{\pm}
    =\gkakko*{r}_{f}^{\pm}
    \coloneqq \pi i \int_{i\infty}^{r}f(z)dz\pm\pi i\int_{i\infty}^{-r}f(z)dz.
\end{align}
It is straightforward to check that $\gkakko*{r}^{+}\in \bbR$ and $\gkakko*{r}^{-}\in i\bbR$. We also set
\begin{align}
    [r]^{+}\coloneqq \frac{1}{\Omega_{f}^{+}}\gkakko*{r}^{+},\quad
    [r]^{-}\coloneqq \frac{1}{i\Omega_{f}^{-}}\gkakko*{r}^{-}.
\end{align}
Then, for any $r \in \bbP^{1}(\bbQ)$, we have $[r]^{\pm} \in \bbQ$ (\textit{cf}. \cite[Chapter I\hspace{-1.2pt}V, \S 2]{Lang:1995}) and $\lkakko*{-r}^{\pm}=\pm\lkakko*{r}^{\pm}$. Note that $\lkakko*{0}^{+}$ is equal to the algebraic part $L(f, 1)/\Omega_{f}^{+}$.

Let $m=q_1\dots q_r$ be the product of distinct odd primes and $M=4^nm$ with $n\in\{0, 1\}$. We write the primitive quadratic Dirichlet character with conductor $M$ as $\chi_M$. Note that for a odd prime number $q$, the sign $\sgn(\chi_q)\coloneqq \chi_q(-1)$ is $+1$ if and only if $q\equiv 1\bmod 4$.

\begin{thm}[cf. {\cite[Theorem 9.9]{manin1}}]\label{thm:ShimuraManin}
    Let $\chi_{M}$ be a quadratic character with conductor $M$ and $\tau(\chi_M)=\sum_{a \, \textrm{mod}\, M}\chi_M(a)e^{2\pi ia/M}$. Then, we have
    \begin{align}
        \tau(\chi_M)\frac{L(f, \chi_{M}, 1)}{\Omega_{f}^{\sgn(\chi_{M})}} =\sum_{k\in (\bbZ\slash M\bbZ)^{\times}} \chi_{M}(k)\left[ \frac{k}{M} \right]^{\sgn(\chi_{M})}.
    \end{align}
    In particular, $L(f, \chi_M, 1)/\Omega_f^{\sgn(\chi_M)}$ is algebraic.
\end{thm}

\section{Twisted \texorpdfstring{$L$}{L}-function and Modular Symbols}\label{sec:TwistLModSymb}

\noindent
In the rest of the paper, we fix a normalized Hecke eigenform $f\in S_{2}(\Gamma_{0}(N))^{\mathrm{new}}$ defined over $\bbQ$. Let $M=4^nm=4^nq_1\cdots q_r$ be an integer with $(M, N)=1$,
where $n\in \{0, 1\}$ and $q_1, \dots, q_r$ are distinct odd primes.
\\

The following lemma is crucial in Proposotion \ref{prop:SumBound1} and Proposition \ref{prop:SumBound2}. These propositions enable us to calculate the left-hand side in equation \eqref{eq:Intro-1}. 

\begin{lem}\label{lem:Integrality}
    For $k\in \bbZ$ with $(k, M)=1$, the following holds.
    \begin{align}
        \nu\skakko*{\lkakko*{\frac{k}{M}}^{+}-\lkakko*{0}^{+}}\in\bbZ, \quad
        \nu\lkakko*{\frac{k}{M}}^{-}\in\bbZ,
    \end{align}
    where $\nu=1$ if $\calL_f$ is rectangular and $\nu=2$ otherwise. 
    In particular, we have
    \begin{align}
        v_2\skakko*{\lkakko*{\frac{k}{M}}^{+}+\lkakko*{\frac{k}{M}}^{-}}
        \geq \min\mkakko*{0, v_2\skakko*{\lkakko*{0}^{+}}}
    \end{align}
\end{lem}
\begin{proof}
    Since $(M, kN)=1$, we can take $a, b\in\bbZ$ such that $aM-bkN=1$.
    Then the matrix
    \begin{align}
        \gamma^{+}\coloneqq
        \begin{pmatrix}
            a & k \\
            bN & M
        \end{pmatrix}
    \end{align}
    belongs to $\Gamma_0(N)$ and $\gamma^{+}\cdot 0=k/M$ holds.
    Similarly, there exists $\gamma^{-}\in\Gamma_0(N)$ such that $\gamma^{-}\cdot 0=-k/M$.
    Therefore it holds that
    \begin{align}
        \Omega_{f}^{+}\lkakko*{\frac{k}{M}}^{+}\pm i\Omega_{f}^{-}\lkakko*{\frac{k}{M}}^{-}
        &=2\pi i\int_{i\infty}^{\gamma^{\pm}\cdot 0}f(z)dz\\
        &=2\pi i\int_{i\infty}^{0}f(z)dz+2\pi i\int_{0}^{\gamma^{\pm}\cdot 0}f(z)dz\\
        &=\Omega_{f}^{+}\lkakko*{0}^{+}+2\pi i\int_{0}^{\gamma^{\pm}\cdot 0}f(z)dz.
    \end{align}
    Thus we have
    \begin{align}
        \Omega_{f}^{+}\skakko*{\lkakko*{\frac{k}{M}}^{+}-\lkakko*{0}^{+}}\pm i\Omega_{f}^{-}\lkakko*{\frac{k}{M}}^{-}
        \in \calL_{f}.
    \end{align}
    and the claim follows by the equality \eqref{eq:Latticef}.
\end{proof}

For each prime $p$ not dividing $N$, the Hecke operator $T_{p}$ acts on modular symbols as follows:
\begin{align}
    T_{p}\lkakko*{r}^{\pm}=\lkakko*{pr}^{\pm}+\sum_{k\in\bbZ/p\bbZ}\lkakko*{\frac{k+r}{p}}^{\pm}.
\end{align}
Moreover, since $f$ is a Hecke eigenform, we see that $T_p\lkakko*{r}^{\pm}=a_p\lkakko*{r}^{\pm}$, where $a_p$ is the $p$-th Fourier coefficient of $f$.

We put
\begin{align}
    \calT_{M}\coloneqq \sum_{k\in(\bbZ/M\bbZ)^\times}\chi_{M}(k)\lkakko*{\frac{k}{M}}^{\sgn(\chi_{M})},
\end{align}
$\calT_{1}\coloneqq \lkakko*{0}^{+}=L(f, 1)/\Omega_{f}^{+}$ and $\calT_{4}:=2\lkakko*{1/4}^-=\tau(\chi_4)L(f,\chi_4,1)\slash\Omega_f^-$.
Moreover, for each positive divisor $D$ of $M$, we set
\begin{align}
    \calT_{D, M}\coloneqq \sum_{k\in(\bbZ/M\bbZ)^\times}\chi_{D}(k)\lkakko*{\frac{k}{M}}^{\sgn(\chi_{D})}.
\end{align}
Obviously, we have $\calT_{M, M}=\calT_{M}$. By Theorem \ref{thm:ShimuraManin}, $\calT_{M}$ is essentially equal to the algebraic part of $L(f, \chi_{M}, 1)$. First, we show that $\calT_{D, M}$ is also essentially
equal to the algebraic part of $L(f, \chi_{D}, 1)$. (See Proposition \ref{prop:TDD}.)

\begin{lem}\label{lem:TDM}
    For an odd prime $q$ dividing $M$ and a positive divisor $D$ of $M/q$, we have
    \begin{align}
        \calT_{D, M}=(a_{q}-2\chi_{D}(q))\calT_{D, \frac{M}{q}}.
    \end{align}
\end{lem}
\begin{proof}
    We can identify $(\bbZ/M\bbZ)^\times$ with the set
    \begin{align}
        \set*{k'\frac{M}{q}+k}{k\in\skakko*{\bbZ/(M/q)\bbZ}^{\times}, \ k'\in\skakko*{\bbZ/q\bbZ}^{\times}}.
    \end{align}
    Therefore it holds that
     \begin{align}
        \calT_{D, M}
        &=\sum_{k\in(\bbZ/(M/q)\bbZ)^{\times}}\sum_{k'\in(\bbZ/q\bbZ)^{\times}}\chi_{D}\skakko*{k}\lkakko*{\frac{k'\frac{M}{q}+k}{M}}^{\sgn(\chi_{D})}\\
        &=\sum_{k\in(\bbZ/(M/q)\bbZ)^{\times}}\skakko*{\sum_{k'\in\bbZ/q\bbZ}\chi_{D}\skakko*{k}\lkakko*{\frac{k'\frac{M}{q}+k}{M}}^{\sgn(\chi_{D})}-\chi_{D}(k)\lkakko*{\frac{k}{M}}^{\sgn(\chi_{D})}}\\
        &=\begin{multlined}[t][12cm]
            \sum_{k\in(\bbZ/(M/q)\bbZ)^{\times}}\chi_{D}(k)\sum_{k'\in\bbZ/q\bbZ}\lkakko*{\frac{k'\frac{M}{q}+k}{M}}^{\sgn(\chi_{D})}\\
            -\sum_{k\in(\bbZ/(M/q)\bbZ)^{\times}}\chi_{D}(qk)\lkakko*{\frac{kq}{M}}^{\sgn(\chi_{D})}.
        \end{multlined}
        \label{eq:TDM-1}
    \end{align}
    Note that the last equality follows from the transformation $k\mapsto qk$ in the second term. We temporarily fix $k\in (\bbZ/(M/q)\bbZ)^{\times}$.
    Considering the action of the Hecke operator $T_{q}$ at the prime $q$, we have
    \begin{align}
        a_{q}\lkakko*{\frac{kq}{M}}^{\sgn(\chi_{D})}
        &=\lkakko*{\frac{kq^2}{M}}^{\sgn(\chi_{D})}+\sum_{k'\in\bbZ/q\bbZ}\lkakko*{\frac{k'\frac{M}{q}+k}{M}}^{\sgn(\chi_{D})}.\label{eq:TDM-2}
    \end{align}
    Taking the sum $\sum_{k\in(\bbZ/(M/q)\bbZ)^{\times}}\chi_{D}(k)$ in the equation \eqref{eq:TDM-2} yields the identity
       \begin{align}
        \begin{multlined}[t][12cm]
            \sum_{k\in(\bbZ/(M/q)\bbZ)^{\times}}\chi_{D}(k)\sum_{k'\in\bbZ/q\bbZ}\lkakko*{\frac{k'\frac{M}{q}+k}{M}}^{\sgn(\chi_{D})}\\
            =(a_{q}-\chi_{D}(q))\sum_{k\in(\bbZ/(M/q)\bbZ)^{\times}}\chi_{D}(k)\lkakko*{\frac{kq}{M}}^{\sgn(\chi_{D})}.\label{eq:TDM-3}
        \end{multlined}
    \end{align}
    Substituting the right-hand side of the equation \eqref{eq:TDM-3} into the right-hand side of the equation \eqref{eq:TDM-1} leads to the claim.
\end{proof}

\begin{prop}\label{prop:TDD}
    For a positive divisor $D$ of $M$, we have
    \begin{align}
        \calT_{D, M}=\calT_{D}\cdot\prod_{q\mid \frac{M}{D}}(a_{q}-2\chi_{D}(q)).
    \end{align}
    In particular, we have
    \begin{align}
        \calT_{1, m}=\calT_{1}\prod_{i=1}^{r}(a_{q_i}-2), \quad
        \calT_{4, 4m}=\calT_{4}\prod_{i=1}^{r}(a_{q_i}-2\chi_{4}(q)).
    \end{align}
\end{prop}
\begin{proof}
    The claim follows by using Lemma \ref{lem:TDM} iteratively.
\end{proof}

In the rest of the paper, for some function $F$, $\sum_{d\mid m}F(d)$ denotes the summation over all positive divisor $d$ of $m$. Moreover, $\sum_{D}'F(D)$ denotes $\sum_{d\mid m}F(d)$ if $n=0$ and $\sum_{d\mid m}F(4d)$ if $n=1$. For example, the sum $\sum_{D}'\calT_{D, M}$ is given by
\begin{align}
    \sideset{}{'}\sum_{D}\calT_{D, M}=\begin{cases}
        \sum_{d\mid m}\calT_{d, m} & (n=0),\\
        \sum_{d\mid m}\calT_{4d, 4m} & (n=1).
    \end{cases}
\end{align}

Next, we calculate a lower bound for $v_2(\sum_{D}'\calT_{D, M})$ in Proposition \ref{prop:SumBound1} and Proposition \ref{prop:SumBound2}. Proposition \ref{prop:SumBound1} gives a general lower bound. Proposition \ref{prop:SumBound2} states it in the case where $\sgn(\chi_{q})$ is $+1$ for all odd prime divisors $q$ of $M$.

\begin{lem}\label{lem:CharacterSum}
    For each $k\in (\bbZ/M\bbZ)^{\times}$, the following holds:
    \begin{align}
        \sideset{}{'}\sum_{D}\chi_{D}(k)=
        \begin{cases}
            \prod_{i=1}^{r}(1+\chi_{q_i}(k)) & (n=0),\\
            \chi_{4}(k)\prod_{i=1}^{r}(1+\chi_{q_i}(k)) & (n=1).
        \end{cases}
    \end{align}
    In particular, we have $v_2(\sum_{D}'\chi_{D}(k))\geq r$.
\end{lem}
\begin{proof}
    One can prove the claim by induction
    on the number of prime divisors of $m$.
\end{proof}

\begin{prop}\label{prop:SumBound1}
    We have
    \begin{align}
        v_2\skakko*{\sideset{}{'}\sum_{D}\calT_{D, M}}
        \geq r+\min\mkakko*{0, v_2\skakko*{\calT_{1}}}.
    \end{align}
\end{prop}
\begin{proof}
    From the facts that $(-1)$-multiplication map on $(\bbZ/M\bbZ)^{\times}$ is bijective
    and $[-r]^{\pm}=\pm[r]^{\pm}$ holds for each $r\in\bbQ$, we see that
    \begin{align}
        &\sum_{k\in(\bbZ/M\bbZ)^\times}\sideset{}{'}\sum_{\substack{D \\ \sgn(\chi_{D})=+}}\chi_{D}(k)\lkakko*{\frac{k}{M}}^{-\sgn(\chi_{D})}=0,\\
        &\sum_{k\in(\bbZ/M\bbZ)^\times}\sideset{}{'}\sum_{\substack{D \\ \sgn(\chi_{D})=-}}\chi_{D}(k)\lkakko*{\frac{k}{M}}^{\sgn(\chi_{D})}
        =0.
    \end{align}
    Therefore we have
    \begin{align}
        \sideset{}{'}\sum_{D}\calT_{D, M}
        &=\sum_{k\in(\bbZ/M\bbZ)^{\times}}\sideset{}{'}\sum_{D}\chi_{D}(k)\lkakko*{\frac{k}{M}}^{\sgn(\chi_{D})}\\
        &=\begin{multlined}[t][11cm]
            \sum_{k\in(\bbZ/M\bbZ)^{\times}}\sideset{}{'}\sum_{\substack{D \\ \sgn(\chi_{D})=+}}\chi_{D}(k)\lkakko*{\frac{k}{M}}^{\sgn(\chi_{D})}\\
            +\sum_{k\in(\bbZ/M\bbZ)^{\times}}\sideset{}{'}\sum_{\substack{D \\ \sgn(\chi_{D})=-}}\chi_{D}(k)\lkakko*{\frac{k}{M}}^{\sgn(\chi_{D})}
        \end{multlined}\\
        &=\begin{multlined}[t][11cm]
            \sum_{k\in(\bbZ/M\bbZ)^{\times}}\sideset{}{'}\sum_{\substack{D \\ \sgn(\chi_{D})=+}}\chi_{D}(k)\lkakko*{\frac{k}{M}}^{+}\\
            +\sum_{k\in(\bbZ/M\bbZ)^{\times}}\sideset{}{'}\sum_{\substack{D \\ \sgn(\chi_{D})=-}}\chi_{D}(k)\lkakko*{\frac{k}{M}}^{-}
        \end{multlined}\\
        &=\sum_{k\in(\bbZ/M\bbZ)^{\times}}\skakko*{\sideset{}{'}\sum_{D}\chi_{D}(k)}\skakko*{\lkakko*{\frac{k}{M}}^{+}+\lkakko*{\frac{k}{M}}^{-}}\label{eq:SumBound1-1}
    \end{align}

    The claim follows from Lemma \ref{lem:Integrality} and Lemma \ref{lem:CharacterSum}.
\end{proof}

\begin{prop}\label{prop:SumBound2}
    Suppose that $\sgn(\chi_{q})$ is $+1$ for all odd prime divisors $q$ of $M$.
    \begin{enumerate}
        \item If $n=0$, then
        \begin{align}
            v_2\skakko*{\sum_{d\mid m}\calT_{d, m}}\geq
            \begin{cases}
                r+1+\min\mkakko*{0, v_2\skakko*{\calT_{1}}} & (\calL_{f}: \text{rectangular}),\\
                r+1+\min\mkakko*{-1, v_2\skakko*{\calT_{1}}} & (\calL_{f}: \text{non-rectangular}).
            \end{cases}
        \end{align}
        \item If $n=1$, then
        \begin{align}
            v_2\skakko*{\sum_{d\mid m}\calT_{4d, 4m}}\geq
            \begin{cases}
                r+1 & (\calL_{f}: \text{rectangular}),\\
                r & (\calL_{f}: \text{non-rectangular}).
            \end{cases}
        \end{align}
    \end{enumerate}
\end{prop}
\begin{proof}
    We only prove the case where $n=1$.
    (For $n=0$, one can prove similarly.) Note that we have the identification
    \begin{align}
    (\bbZ\slash 4m\bbZ)^{\times}=\set*{4k+m,4k-m}{k\in (\bbZ\slash m\bbZ)^{\times}}.
    \end{align}
    From the equation \eqref{eq:SumBound1-1}, we have
    \begin{align}
        \sum_{d\mid m}\calT_{4d, 4m}
        &=\sum_{\substack{k=1 \\ (k, 4m)=1}}^{4m}\skakko*{\sum_{d\mid m}\chi_{4d}(k)}\skakko*{\lkakko*{\frac{k}{4m}}^{+}+\lkakko*{\frac{k}{4m}}^{-}}\\
        &=
        \begin{multlined}[t][11cm]
            \sum_{\substack{k=1 \\ (k, 4m)=1}}^{2m}\skakko*{\sum_{d\mid m}\chi_{4d}(k)}\skakko*{\lkakko*{\frac{k}{4m}}^{+}+\lkakko*{\frac{k}{4m}}^{-}}\\
            +\sum_{\substack{k=1 \\ (k, 4m)=1}}^{2m}\skakko*{\sum_{d\mid m}\chi_{4d}(4m-k)}\skakko*{\lkakko*{\frac{4m-k}{4m}}^{+}+\lkakko*{\frac{4m-k}{4m}}^{-}}
        \end{multlined}
        \\
        &=
        \begin{multlined}[t][11cm]
            \sum_{\substack{k=1 \\ (k, 4m)=1}}^{2m}\skakko*{\sum_{d\mid m}\chi_{4d}(k)}\skakko*{\lkakko*{\frac{k}{4m}}^{+}+\lkakko*{\frac{k}{4m}}^{-}}\\
            +\sum_{\substack{k=1 \\ (k, 4m)=1}}^{2m}\skakko*{\sum_{d\mid m}\chi_{4d}(-k)}\skakko*{\lkakko*{\frac{k}{4m}}^{+}-\lkakko*{\frac{k}{4m}}^{-}}.
        \end{multlined}
    \end{align}    
    The claim follows from Lemma \ref{lem:Integrality} and Lemma \ref{lem:CharacterSum}.
\end{proof}

Finally, we calculate a lower bound for $\calT_{1}$ and $\calT_{4}$ in Proposition \ref{prop:BeforeTwist}. This proposition is related to whether the assumptions of the main theorem are empty or not. The details are mentioned in Remark \ref{rem:Assumption}.

\begin{lem}\label{lem:BeforeTwist}
    The rational numbers
    \begin{align}
        \calT_{1, m}-\#(\bbZ/m\bbZ)^\times\cdot\calT_{1}, \quad
        \calT_{4, 4m}
    \end{align}
    belong to $2\bbZ$ (resp. $\bbZ$) if $\calL_{f}$ is rectangular (resp. non-rectangular).
\end{lem}
\begin{proof}
    First, we calculate $\calT_{1, m}-\#(\bbZ/m\bbZ)^\times\cdot\calT_{1}$.
    Dividing $\calT_{1, m}$ into two summations, we have
    \begin{align}
        \calT_{1, m}
        &=\sum_{k\in(\bbZ/m\bbZ)^\times}\lkakko*{\frac{k}{m}}^{+}\\
        &=\sum_{\substack{k=1 \\ (k, m)=1}}^{(m-1)/2}\lkakko*{\frac{k}{m}}^{+}+\sum_{\substack{k=1 \\ (k, m)=1}}^{(m-1)/2}\lkakko*{\frac{m-k}{m}}^{+}\\
        &=2\sum_{\substack{k=1 \\ (k, m)=1}}^{(m-1)/2}\lkakko*{\frac{k}{m}}^{+}.
    \end{align}
    Therefore it holds
    \begin{align}
        \calT_{1, m}-2\sum_{\substack{k=1 \\ (k, m)=1}}^{(m-1)/2}\lkakko*{0}^{+}
        =2\sum_{\substack{k=1 \\ (k, m)=1}}^{(m-1)/2}\skakko*{\lkakko*{\frac{k}{m}}^{+}-\lkakko*{0}^{+}}
    \end{align}
    and the claim follows from Lemma \ref{lem:Integrality}.
    Second, we calculate $\calT_{4, 4m}$. By using the identification
    \begin{align}
        (\bbZ/4m\bbZ)^\times=\set*{4k+m, 4k-m}{k\in(\bbZ/m\bbZ)^\times},
    \end{align}
    we have
    \begin{align}
        \calT_{4, 4m}
        &=\sum_{k\in(\bbZ/4m\bbZ)^\times}\chi_{4}(k)\lkakko*{\frac{k}{4m}}^{-}\\
        &=\sum_{k\in(\bbZ/m\bbZ)^\times}\skakko*{\chi_{4}(4k+m)\lkakko*{\frac{4k+m}{4m}}^{-}+\chi_{4}(4k-m)\lkakko*{\frac{4k-m}{4m}}^{-}}\\
        &=\chi_{4}(m)\sum_{k\in(\bbZ/m\bbZ)^\times}\skakko*{\lkakko*{\frac{4k+m}{4m}}^{-}-\lkakko*{\frac{4k-m}{4m}}^{-}}.
    \end{align}
    Since
    \begin{align}
        -\sum_{k\in(\bbZ/m\bbZ)^\times}\lkakko*{\frac{4k-m}{4m}}^{-}
        =\sum_{k\in(\bbZ/m\bbZ)^\times}\lkakko*{\frac{-4k+m}{4m}}^{-}
        =\sum_{k\in(\bbZ/m\bbZ)^\times}\lkakko*{\frac{4k+m}{4m}}^{-},
    \end{align}
    we have
    \begin{align}
        \calT_{4, 4m}=2\chi_{4}(m)\sum_{k\in(\bbZ/m\bbZ)^\times}\lkakko*{\frac{4k+m}{4m}}^{-}.\label{eq:BeforeTwist-1}
    \end{align}
    The claim follows from here.
\end{proof}

We will now show how an even Fourier coefficient forces extra 2-divisibility.

\begin{lem}\label{lem:BeforeTwist2}
    If there exists an odd prime $q$ dividing $m$ such that $a_q$ is even, then we have
    \begin{align}
        \calT_{4, 4m}\in \begin{cases}
            4\bbZ & (\calL_f: \text{rectangular}),\\
            2\bbZ & (\calL_f: \text{non-rectangular}).
        \end{cases}
    \end{align}
\end{lem}
\begin{proof}
    By the equation \eqref{eq:BeforeTwist-1}, it is sufficient to show
    \begin{align}
        \sum_{k\in\bbZ/m\bbZ}\lkakko*{\frac{4k+m}{4m}}^{-}\in\begin{cases}
            2\bbZ & (\calL_f: \text{rectangular}),\\
            \bbZ & (\calL_f: \text{non-rectangular}).
        \end{cases}\label{eq:BeforeTwist2-1}
    \end{align} 
    Let $q$ be a prime dividing $m$ such that $a_q$ is even. Recall the identification
    \begin{align}
        (\bbZ/m\bbZ)^\times=\set*{k'\frac{m}{q}+k}{k\in (\bbZ/(m/q)\bbZ)^\times, \ k'\in(\bbZ/q\bbZ)^\times}.
    \end{align}
    Then we see that
    \begin{align}
        \sum_{k\in(\bbZ/m\bbZ)^\times}\lkakko*{\frac{4k+m}{4m}}^{-}
        =\sum_{k'\in(\bbZ/q\bbZ)^\times}\sum_{k\in(\bbZ/(m/q)\bbZ)^\times}\lkakko*{\frac{4(k'\frac{m}{q}+k)+m}{4m}}^{-}.\label{eq:BeforeTwist2-2}
    \end{align}
    We temporarily fix $k\in (\bbZ/(m/q)\bbZ)^\times$ and calculate as follows:
    \begin{align}
        a_{q}\lkakko*{\frac{(4k+m)q}{4m}}^{-}
        &=\lkakko*{\frac{(4k+m)q^2}{4m}}^{-}+\sum_{k'\in\bbZ/q\bbZ}\lkakko*{\frac{k'+\frac{(4k+m)q}{4m}}{q}}^{-}\\
        &=\lkakko*{\frac{4kq^2+m}{4m}}^{-}+\lkakko*{\frac{4k+m}{4m}}^{-}+\sum_{k'\in(\bbZ/q\bbZ)^\times}\lkakko*{\frac{4(k'\frac{m}{q}+k)+m}{4m}}^{-}.\label{eq:BeforeTwist2-3}
    \end{align}
    Note that $[q^2/4]^{\pm}=[1/4]^{\pm}$ since $q^2\equiv 1\bmod 4$.
    Taking the sum $\sum_{k\in(\bbZ/(m/q)\bbZ)^\times}$ in the equation \eqref{eq:BeforeTwist2-3} yields
    \begin{align}
        %\begin{multlined}[t][12cm]
            &\sum_{k'\in(\bbZ/q\bbZ)^\times}\sum_{k\in(\bbZ/(m/q)\bbZ)^\times}\lkakko*{\frac{4(k'\frac{m}{q}+k)+m}{4m}}^{-}\\
            &=\sum_{k\in(\bbZ/(m/q)\bbZ)^\times}\mkakko*{a_{q}\lkakko*{\frac{(4k+m)q}{4m}}^{-}-\lkakko*{\frac{4k+m}{4m}}^{-}}-\sum_{k\in(\bbZ/(m/q)\bbZ)^\times}\lkakko*{\frac{4kq^2+m}{4m}}^{-}\\
            &=\sum_{k\in(\bbZ/(m/q)\bbZ)^\times}\mkakko*{a_{q}\lkakko*{\frac{(4k+m)q}{4m}}^{-}-2\lkakko*{\frac{4k+m}{4m}}^{-}},\label{eq:BeforeTwist2-4}
        %\end{multlined}
    \end{align}
    where the last equality is obtained by using the transformation $k\mapsto q^{-2}k$ on the latter sum. Therefore, we obtain
    \begin{align}
        \begin{multlined}[t][12cm]
            \sum_{k\in(\bbZ/m\bbZ)^\times}\lkakko*{\frac{4k+m}{4m}}^{-}\\
            =a_{q}\sum_{k\in(\bbZ/(m/q)\bbZ)^\times}\lkakko*{\frac{(4k+m)q}{4m}}^{-}-2\sum_{k\in(\bbZ/(m/q)\bbZ)^\times}\lkakko*{\frac{4k+m}{4m}}^{-}.
        \end{multlined}
    \end{align}
    By the assumption, $a_{q}$ is an even integer. Therefore, the lemma holds from Lemma \ref{lem:Integrality}.
\end{proof}

By the above arguments, the 2-adic valuations of $a_q-(q+1)$ and $a_q-2\chi_4(q)$ for any prime $q$ dividing $m$ have the following lower bounds.

\begin{prop}\label{prop:BeforeTwist}
    We put $\frakw_m$ as $0$ if $\frakv_m=0$ and 1 otherwise.
    Then, for any prime $q$ dividing $m$, the following holds.
    \begin{enumerate}
        \item If $\calL_{f}$ is rectangular, then
        \begin{align}
            \hspace{-10pt}
            v_2(a_{q}-(q+1))\geq \frakw_m+1-v_2(\calT_{1}),\quad
            v_2(a_{q}-2\chi_{4}(q))\geq \frakw_m+1-v_2(\calT_{4})
        \end{align}
        \item If $\calL_{f}$ is non-rectangular, then
        \begin{align}
            v_2(a_{q}-(q+1))\geq \frakw_m-v_2(\calT_{1}),\quad
            v_2(a_{q}-2\chi_{4}(q))\geq \frakw_m-v_2(\calT_{4})
        \end{align}
    \end{enumerate}
\end{prop}
\begin{proof}
    By Proposition \ref{prop:TDD}, we have
    \begin{align}
        &\calT_{1, q}-\#(\bbZ/q\bbZ)^\times\calT_{1}
        =(a_{q}-(q+1))\calT_{1}, \quad
        \calT_{4, 4q}=(a_{q}-2\chi_{4}(q))\calT_{4}.
    \end{align}
    The claim follows from Lemma \ref{lem:BeforeTwist} and Lemma \ref{lem:BeforeTwist2}.
\end{proof}

\section{The 2-adic valuations of the algebraic parts}\label{sec:TwoAdicVal}

Let $M=4^nm$ be a positive integer with $(M, N)=1$, where $n\in\{0, 1\}$ and $m$ is a square-free odd integer. We define $r(m)$ as the number of the prime divisors of $m$. We set
\begin{align}
    \frakv_m=\min_{q\mid m: \text{prime}}\mkakko*{v_{2}(a_{q}-2)}
\end{align}
and
\begin{align}
    \frakw_m=\begin{cases}
        0 & (\frakv_m=0),\\
        1 & (\frakv_m\geq 1).
    \end{cases}
\end{align}
Moreover, for an integer $i$, we define
\begin{align}
    &\calS_{i}^{\pm}=\set*{q: \text{odd prime}}{q\nmid N, \sgn(\chi_{q})=\pm 1, v_2(a_{q}-2)=i}
\end{align}
and set $\calS_{i}=\calS_{i}^{+}\cup\calS_{i}^{-}$. Note that if $\calS_{i}^{\pm}$ is non-empty, by Chebotarev's density theorem, $\calS_{i}^{\pm}$ has positive density.

In the following, we state Theorem \ref{thm:MainTheorem} in four cases, according to whether $\calL_f$ is rectangular or not, and whether $n=0$ or not. Since proofs are the same in all cases, we only prove Theorem \ref{thm:MainTheorem} for the case where $\calL_f$ is rectangular and $n=1$ in Theorem \ref{thm:MainThm3}. In the following, $\delta_{i, j}$ denotes the Kronecker delta.

\begin{thm}[Rectangular, $n=0$]\label{thm:MainThm1}
    We assume that $\calL_{f}$ is rectangular and $n=0$ and write $m=q_1\cdots q_{r(m)}$.
    \begin{enumerate}
        \item If $q_i\in \calS_{0}\cup\calS_{1}\cup\calS_{2}$ for all $i$, then
        \begin{align}
        \label{eq:MainThm1-1}\\
           \hspace{-60pt} v_2\skakko*{\frac{L(f, \chi_{m}, 1)}{\Omega_{f}^{\sgn(\chi_{m})}}}
            \geq \frakw_m\cdot r(m)+\min\mkakko*{\delta_{\frakv_m, 0}, v_2\skakko*{\frac{L(f, 1)}{\Omega_{f}^{+}}}}.
        \end{align}
        \item If $q_i\in \calS_{0}^{+}\cup\calS_{1}^{+}\cup\calS_{2}^{+}$ for all $i$, then
        \begin{align}
        \label{eq:MainThm1-2}\\
        \begin{multlined}[t][12cm]
            v_2\skakko*{\frac{L(f, \chi_{m}, 1)}{\Omega_{f}^{\sgn(\chi_{m})}}}
            \geq \frakw_m\cdot r(m)+\min\mkakko*{1+\delta_{\frakv_m, 0}, v_2\skakko*{\frac{L(f, 1)}{\Omega_{f}^{+}}}+\delta_{\frakv_m, 2}}.\end{multlined}
        \end{align}
        In particular, if $v_2(L(f, 1)/\Omega_{f}^{+})=1$ and $q_i\in \calS_{0}^{+}$ for all $i$, or $v_2(L(f, 1)/\Omega_{f}^{+})<1$ and $q_{i}\in \calS_{1}^{+}$ for all $i$, then the equality in the inequality \eqref{eq:MainThm1-2} holds.
    \end{enumerate}
\end{thm}

\begin{thm}[Non-rectangular, $n=0$]\label{thm:MainThm2}
    We assume that $\calL_{f}$ is non-rectangular and $n=0$ and write $m=q_1\cdots q_{r(m)}$.
    \begin{enumerate}
        \item If $q_i\in\calS_{0}\cup\calS_{1}\cup\calS_{2}$ for all $i$, then
        \begin{align}
        \label{eq:MainThm2-1}\\
           \hspace{-40pt}  v_2\skakko*{\frac{L(f, \chi_{M}, 1)}{\Omega_{f}^{\sgn(\chi_{M})}}}
            \geq \frakw_m\cdot r(m)+\min\mkakko*{\delta_{\frakv_m, 0}, v_2\skakko*{\frac{L(f, 1)}{\Omega_{f}^{+}}}}.
        \end{align}
        In particular, if $v_2(L(f, 1)/\Omega_f)=0$ and $q_i\in\calS_{0}$ for all $i$, then the equality in the inequality \eqref{eq:MainThm2-1} holds.
        \item If $q_i\in\calS_{0}^{+}\cup\calS_{1}^{+}\cup\calS_{2}^{+}$ for all $i$, then
        \begin{align}
             \begin{multlined}[t][12cm]
             \label{eq:MainThm2-2}\\ %\vspace{-2pt}\\ \
        \hspace{-40pt} v_2\skakko*{\frac{L(f, \chi_{M}, 1)}{\Omega_{f}^{\sgn(\chi_{M})}}}
            \geq \frakw_m\cdot r(m)+\min\mkakko*{\delta_{\frakv_m, 0}, v_2\skakko*{\frac{L(f, 1)}{\Omega_{f}^{+}}}+\delta_{\frakv_m, 2}}.
         \end{multlined}
        \end{align}
        In particular, if $v_2(L(f, 1)/\Omega_{f}^{+})=0$ and $q_i\in \calS_{0}^{+}$ for all $i$, or $v_2(L(f, 1)/\Omega_{f}^{+})<0$ and $q_i\in \calS_{1}^{+}$, then the equality in the inequality \eqref{eq:MainThm2-2} holds.
    \end{enumerate}
\end{thm}

\begin{thm}[Rectangular, $n=1$]\label{thm:MainThm3}
    We assume that $\calL_{f}$ is rectangular and $n=1$ and write $m=q_1\cdots q_{r(m)}$.
    \begin{enumerate}
        \item If $q_i\in\calS_{0}\cup\calS_{1}\cup\calS_{2}$ for all $i$, then 
        \begin{align}
            \begin{multlined}[t][12cm]
                v_2\skakko*{\frac{L(f, \chi_{M}, 1)}{\Omega_{f}^{\sgn(\chi_{M})}}}\geq \frakw_m\cdot r(m)\\
                +\min\mkakko*{-1+\delta_{\frakv_m, 0}, -1+\delta_{\frakv_m, 0}+v_2\skakko*{\frac{L(f, 1)}{\Omega_f^+}}, v_2\skakko*{\frac{L(f, \chi_4, 1)}{\Omega_f^-}}+\delta_{\frakv_m, 2}}.
            \end{multlined}\label{eq:MainThm3-1}
        \end{align}
        \item If $q_i\in\calS_{0}^{+}\cup\calS_{1}^{+}\cup\calS_{2}^{+}$ for all $i$, then
      \begin{align}
         \begin{multlined}[t][12cm]
    \label{eq:MainThm3-2}\vspace{-2pt}\\ 
   \hspace{-30pt} v_2\skakko*{\frac{L(f, \chi_{M}, 1)}{\Omega_{f}^{\sgn(\chi_{M})}}}\geq \frakw_m\cdot r(m)+\min\mkakko*{\delta_{\frakv_m, 0}, v_2\skakko*{\frac{L(f, \chi_{4}, 1)}{\Omega_{f}^{-}}}+\delta_{\frakv_m, 2}}.
 \end{multlined}
        \end{align}
  
        In particular, if $v_2(L(f, \chi_{4}, 1)/\Omega_{f}^{-})=0$ and $q_i\in\calS_{0}^{+}$ for all $i$, then the equality in the inequality \eqref{eq:MainThm3-2} holds.
    \end{enumerate}
\end{thm}

\begin{thm}[Non-rectangular, $n=1$]\label{thm:MainThm4}
    Assume that $\calL_{f}$ is non-rectangular and $n=1$ and write $m=q_1\cdots q_{r(m)}$.
    \begin{enumerate}
        \item If $q_i\in \calS_{0}\cup\calS_{1}\cup\calS_{2}$ for all $i$, then
        \begin{align}
            \begin{multlined}[t][12cm]
            v_2\skakko*{\frac{L(f, \chi_{M}, 1)}{\Omega_{f}^{\sgn(\chi_{M})}}}\geq \frakw_m\cdot r(m)\\
            +\min\mkakko*{-1+\delta_{\frakv_m, 0}, -1+\delta_{\frakv_m, 0}+v_2\skakko*{\frac{L(f, 1)}{\Omega_{f}^{+}}}, v_2\skakko*{\frac{L(f, \chi_{4}, 1)}{\Omega_{f}^{-}}}+\delta_{\frakv_m, 2}}.
            \end{multlined}\label{eq:MainThm4-1}
        \end{align}
        In particular, if $v_2(L(f, \chi_{4}, 1)/\Omega_{f}^{-})=-1$ and $q_i\in \calS_{0}$ for all $i$, then the equality in the inequality \eqref{eq:MainThm4-1} holds.
        \item If $q_i\in \calS_{0}^{+}\cup\calS_{1}^{+}\cup\calS_{2}^{+}$ for all $i$, then
        \begin{align}
            v_2\skakko*{\frac{L(f, \chi_{M}, 1)}{\Omega_{f}^{\sgn(\chi_{M})}}}
            \geq \frakw_m\cdot r(m)+\min\mkakko*{-1+\delta_{\frakv_m, 0}, v_2\skakko*{\frac{L(f, \chi_{4}, 1)}{\Omega_{f}^{-}}}+\delta_{\frakv_m, 2}}.\label{eq:MainThm4-2}
        \end{align}
        In particular, if $v_2(L(f, \chi_{4}, 1)/\Omega_{f}^{-})=-1$ and $q_i\in \calS_{0}^{+}$ for all $i$, then the equality in the inequality \eqref{eq:MainThm4-2} holds.
    \end{enumerate}
\end{thm}

%\begin{rem}
%    In Theorem \ref{thm:MainThm1}-\ref{thm:MainThm4}, if there exists at least one prime $q_i$ such that $a_{q_i}$ is odd, then we can remove the condition that $q_i\in \calS_0\cup\calS_1\cup\calS_2$ for all $i$. Otherwise, we can also remove the condition. However, these lower bounds may not be sharp.
%\end{rem}

\begin{rem}\label{rem:Assumption}
    In Theorem \ref{thm:MainThm1}-(i), we can show that the equality holds when $v_2(\calT_{1})<1$ and $q_i\in\calS_{0}$ for all $i$. However, Proposition \ref{prop:BeforeTwist} shows that such a condition does not hold. This means that the equality
    \begin{align}
        v_2\skakko*{\frac{L(f, \chi_m, 1)}{\Omega_{f}^{\sgn(\chi_m)}}}=\frakw_m\cdot r(m)+\min\mkakko*{\delta_{\frakv_m, 0}, v_2\skakko*{\frac{L(f, 1)}{\Omega_{f}^{+}}}}
    \end{align}
    is not necessarily determined only by $v_2(\calT_{1})$ and $v_2(a_{q_i}-2)$.
    \textsc{Table} \ref{tab:Assumptions} shows, for each theorem, the conditions under which the equality holds and the conditions under which is non-empty, taking Proposition \ref{prop:BeforeTwist} into account. (The non-empty conditions are the same as those described in the claim for each theorem.)

 \begin{table}[ht]
        \caption{Equality condition and non-empty condition}
        \label{tab:Assumptions}
        \centering
        \begin{tabular}{ccc}
            Inequality & Equality Condition & Non-empty Condition \\ \hline \hline
            \eqref{eq:MainThm1-1} & $v_2(\calT_{1})<1, q_i\in\calS_{0}$ & - \\ \hline
            \eqref{eq:MainThm1-2} & \begin{tabular}{c} $v_2(\calT_{1})<2, q_i\in\calS_{0}^{+}$ \\ $v_2(\calT_{1})<1, q_i\in \calS_{1}^{+}$ \end{tabular} & \begin{tabular}{c} $v_2(\calT_{1})=1, q_i\in \calS_{0}^{+}$ \\ $v_2(\calT_{1})<1, q_i\in \calS_{1}^{+}$ \end{tabular} \\ \hline
            \eqref{eq:MainThm2-1} & $v_2(\calT_{1})<1, q_i\in\calS_{0}$ & $v_2(\calT_{1})=0, q_i\in\calS_{0}$ \\ \hline
            \eqref{eq:MainThm2-2} & \begin{tabular}{c} $v_2(\calT_{1})<1, q_i\in \calS_{0}^{+}$ \\ $v_2(\calT_{1})<0, q_i\in\calS_{1}^{+}$ \end{tabular}& \begin{tabular}{c} $v_2(\calT_{1})=0, q_i\in\calS_{0}^{+}$ \\ $v_2(\calT_{1})<0, q_i\in\calS_{1}^{+}$ \end{tabular} \\ \hline
            \eqref{eq:MainThm3-1} & \begin{tabular}{c} $v_2(\calT_{4})<\min\mkakko*{1, 1+v_2(\calT_{1})}, q_i\in\calS_{0}$ \\ $v_2(\calT_{4})<\min\mkakko*{0, v_2(\calT_{1})}, q_i\in\calS_{1}$ \end{tabular}& - \\ \hline
            \eqref{eq:MainThm3-2} & \begin{tabular}{c} $v_2(\calT_{4})<2, q_i\in\calS_{0}^{+}$ \\ $v_2(\calT_{4})<1, q_i\in\calS_{1}^{+}$ \end{tabular} & $v_2(\calT_{4})=1, q_i\in\calS_{0}^{+}$ \\ \hline
            \eqref{eq:MainThm4-1} & \begin{tabular}{c} $v_2(\calT_{4})<\min\mkakko*{1, 1+v_2(\calT_{1})}, q_i\in \calS_{0}$ \\ $v_2(\calT_{4})<\min\mkakko*{0, v_2(\calT_{1})}, q_i\in\calS_{1}$ \end{tabular}& $v_2(\calT_{4})=0, q_i\in\calS_{0}$ \\ \hline
            \eqref{eq:MainThm4-2} & \begin{tabular}{c} $v_2(\calT_{4})<1, q_i\in\calS_{0}^{+}$ \\ $v_2(\calT_{4})<0, q_i\in\calS_{1}^{+}$ \end{tabular} & $v_2(\calT_{4})=0, q_i\in\calS_{0}^{+}$ \\
        \end{tabular}
    \end{table}
\end{rem}
\newpage
\begin{proof}[Proof of Theorem \ref{thm:MainThm3}]
    We only prove (ii) in the theorem. For the claim (i), we can prove it similarly.
    Note that the inequality \eqref{eq:MainThm3-2} is equivalent to
    \begin{align}
        v_2(\calT_{4m})\geq \frakw_m\cdot r(m)+\min\mkakko*{1+\delta_{\frakv_m, 0}, v_2(\calT_{4})+\delta_{\frakv_m, 2}}.\label{eq:MainThm3-3}
    \end{align}
    We show the claim by induction on the number of the prime divisors of $m$.
    
    We suppose $\sgn(\chi_{q_i})=+1$ for all $i$. Then, from Proposition \ref{prop:SumBound2}, we have the following estimation:
    \begin{align}
        v_2\skakko*{\sum_{d\mid m}\calT_{4d, 4m}}\geq r(m)+1.\label{eq:MainThm3-4}
    \end{align}
    Moreover, we see that
    \begin{align}
        \frakw_{q_1}+\min\mkakko*{1+\delta_{\frakv_{q_1}, 0}, v_2(\calT_4)+\delta_{\frakv_{q_1}, 2}}=\begin{cases}
            \min\mkakko*{2, v_2(\calT_4)} & (\frakv_{q_1}=0),\\
            1+\min\mkakko*{1, v_2(\calT_4)} & (\frakv_{q_1}=1),\\
            2+\min\mkakko*{0, v_2(\calT_4)} & (\frakv_{q_1}=2).
        \end{cases}
    \end{align}
    We consider the case $r(m)=1$. By the inequality \eqref{eq:MainThm3-4}, we have
    \begin{align}
        v_2\skakko*{\sum_{d\mid q_1}\calT_{4d, 4q_1}}
        \geq 2
        \geq \frakw_{q_1}+\min\mkakko*{1+\delta_{\frakv_{q_1}, 0}, v_2(\calT_4)+\delta_{\frakv_{q_1}, 2}}.\label{eq:MainThm3-5}
    \end{align}
    By Proposition \ref{prop:TDD}, we see that
    \begin{align}
        v_2\skakko*{\calT_{4, 4q_{1}}}
        &=v_2(a_{q_1}-2\chi_4(q_1))+v_2(\calT_{4})\\
        &\geq \frakv_{q_1}+v_2\skakko*{\calT_{4}}\label{eq:MainThm3-5.5}\\
        &\geq \frakw_{q_1}+\min\mkakko*{1+\delta_{\frakv_{q_1}, 0}, v_2(\calT_4)+\delta_{\frakv_{q_1}, 2}}.\label{eq:MainThm3-6}
    \end{align}
    Note that the inequality \eqref{eq:MainThm3-5.5} follows from $q_i\in \calS_{0}\cup\calS_{1}\cup\calS_{2}$ for all $i$. Therefore, from the inequality \eqref{eq:MainThm3-5} and the inequality \eqref{eq:MainThm3-6}, we obtain
    \begin{align}
        v_2\skakko*{\calT_{4q_1}}
        &=v_2\skakko*{\sum_{d\mid q_{1}}\calT_{4d, 4q_{1}}-\calT_{4, 4q_{1}}}\\
        &\geq \min\mkakko*{v_2\skakko*{\sum_{d\mid q_{1}}\calT_{4d, 4q_{1}}}, v_2\skakko*{\calT_{4, 4q_{1}}}}\\
        &\geq \frakw_{q_1}+\min\mkakko*{1+\delta_{\frakv_{q_1}, 0}, v_2\skakko*{\calT_{4}}+\delta_{\frakv_{q_1}, 2}}.
    \end{align}
    Moreover, if $v_2(\calT_{4})=1$ and $\frakv_{q_1}=0$, then the left-hand side is strictly larger than the right-hand side in the inequality \eqref{eq:MainThm3-5}, and the both sides are equal in the inequality \eqref{eq:MainThm3-6}. Therefore, we have $v_2(\sum_{d\mid q_1}\calT_{4d, 4q_1})>v_2(\calT_{4, 4q_1})$ and
    \begin{align}
        v_2\skakko*{\calT_{4q_1}}
        &=v_2\skakko*{\sum_{d\mid q_{1}}\calT_{4d, 4q_{1}}-\calT_{4, 4q_{1}}}\\
        &=v_2\skakko*{\calT_{4, 4q_{1}}}\\
        &=\frakw_{q_1}+\min\mkakko*{1+\delta_{\frakv_{q_1}, 0}, v_2\skakko*{\calT_{4}}+\delta_{\frakv_{q_1}, 2}}.
    \end{align}
    Thus, the claim holds for $r(m)=1$. Suppose it is true for $1, \dots, r(m)-1$. Then, by the inequality \eqref{eq:MainThm3-4}, it holds that
    \begin{align}
        v_2\skakko*{\sum_{d\mid m}\calT_{4d, 4m}}
        &\geq r(m)+1\\
        &\geq \frakw_m\cdot r(m)+\min\mkakko*{1+\delta_{\frakv_m, 0}, v_2(\calT_4)+\delta_{\frakv_m, 2}}\label{eq:MainThm3-7}
    \end{align}
    and by Proposition \ref{prop:TDD}, we see that
    \begin{align}
        v_2\skakko*{\calT_{4, 4m}}
        &=\sum_{i=1}^{r(m)}v_2\skakko*{a_{q_i}-2\chi_4(q_i)}+v_2(\calT_{4})\\
        &\geq \sum_{i=1}^{r(m)}v_2\skakko*{a_{q_i}-2}+v_2(\calT_{4})\\
        &\geq \frakv_{m}\cdot r(m)+v_2(\calT_{4})\\
        &\geq \frakw_m\cdot r(m)+\min\mkakko*{1+\delta_{\frakv_{m}, 0}, v_2(\calT_{4})+\delta_{\frakv_m, 2}}.\label{eq:MainThm3-8}
    \end{align}
    For a divisor $d$ of $m$ with $d\neq 1, m$, by Proposition \ref{prop:TDD} and the induction hypothesis, it holds that
    \begin{align}
        v_2\skakko*{\calT_{4d, 4m}}
        &=v_2(\calT_{4d})+\sum_{q\mid \frac{m}{d}}v_2(a_{q}-2\chi_d(q))\\
        &\geq v_2(\calT_{4d})+\sum_{q\mid \frac{m}{d}}v_2(a_{q}-2)\\
        &\geq \frakw_d\cdot r(d)+\min\mkakko*{1+\delta_{\frakv_{d}, 0}, v_2(\calT_{4})+\delta_{\frakv_d, 2}}+\frakv_m\cdot (r(m)-r(d))\\
        &\geq \frakw_m\cdot r(m)+\min\mkakko*{1+\delta_{\frakv_{m}, 0}, v_2(\calT_{4})+\delta_{\frakv_m, 2}}.\label{eq:MainThm3-9}
    \end{align}
    Therefore, we have
    \begin{align}
        v_2\skakko*{\calT_{4m}}
        &=v_2\skakko*{\sum_{d\mid m}\calT_{4d, 4m}-\calT_{4, 4m}-\sum_{\substack{d\mid m \\ d\neq 1, m}}\calT_{4d, 4m}}\\
        &=\min\mkakko*{v_2\skakko*{\sum_{d\mid m}\calT_{4d, 4m}}, v_2\skakko*{\calT_{4, 4m}}, v_2\skakko*{\sum_{\substack{d\mid m \\ d\neq 1, m}}\calT_{4d, 4m}}}\\
        &\geq \frakw_m\cdot r(m)+\min\mkakko*{1+\delta_{\frakv_{m}, 0}, v_2(\calT_{4})+\delta_{\frakv_m, 2}}.\label{eq:MainThm3-10}
    \end{align}
    Moreover, if $v_2(\calT_{4})=1$ and $\frakv_d=0$ for a divisor $d$ of $m$ with $d\neq 1, m$, then the left-hand side is strictly larger than the right-hand side in the inequality \eqref{eq:MainThm3-7}, and both sides are equal in the inequality \eqref{eq:MainThm3-8} and the inequality \eqref{eq:MainThm3-9}. In addition, since the number of divisors of $m$ is even, we see that
    \begin{align}
        v_2\skakko*{\sum_{\substack{d\mid m \\ d\neq 1, m}}\calT_{4d, 4m}}
        &\geq 1+\frakv_m\cdot r(m)+\min\mkakko*{1+\delta_{\frakv_{m}, 0}, v_2(\calT_{4})+\delta_{\frakv_m, 2}}\\
        &>\frakw_m\cdot r(m)+\min\mkakko*{1+\delta_{\frakv_{m}, 0}, v_2(\calT_{4})+\delta_{\frakv_m, 2}}.
    \end{align}
    Thus, we have
    \begin{align}
        v_2\skakko*{\calT_{4, 4m}}<v_2\skakko*{\sum_{d\mid m}\calT_{4d, 4m}}, v_2\skakko*{\sum_{\substack{d\mid m \\ d\neq 1, m}}\calT_{4d, 4m}}.
    \end{align}
    Therefore the equality in the inequality \eqref{eq:MainThm3-10} holds. This completes the proof.
\end{proof}

\section{Numerical Example}\label{sec:Example}

\noindent
In this section, we give examples for the case $N=34$ in Theorem \ref{thm:MainThm1} and for the case $N=37$ in Theorem \ref{thm:MainThm3}.

\subsection{Example of Theorem \ref{thm:MainThm1}}

Let $f_1\in S_2(\Gamma_0(34))^{\mathrm{new}}$ be the normalized Hecke eigenform
\begin{align}
    f_1=q + q^2 - 2q^3 + q^4 - 2q^6 - 4q^7 + q^8 + q^9 + O(q^{10})
\end{align}
By computation, we see that $\calL_{f_1}$ is rectangular, $v_2(L(f_1, 1)/\Omega_{f_1}^+)=0$ and the global minimal Weierstrass model of $E_{f_1}$ is given by the equation $y^2 + xy = x^3 - 3x + 1$ with Cremona label \texttt{34a1}. In \textsc{Table} \ref{tab:table1}, we give the comparison between the 2-adic valuation of the algebraic part $v_2\skakko*{L(f_1, \chi_m, 1)/\Omega_{f_1}^{\sgn(\chi_m)}}$ and the lower bound in Theorem \ref{thm:MainThm1} for various integers $m$ with $r(m)=1$. Here, items for which the equality condition holds are colored by red, that is, the items for $m=q\in S_{1}^{+}$.

\subsection{Example of Theorem \ref{thm:MainThm3}}

Let $f_2\in S_2(\Gamma_0(37))^{\mathrm{new}}$ be the normalized Hecke eigenform
\begin{align}
    f_2=q-2q^2-3q^3+2q^4-2q^5+6q^6-q^7+O(q^8).
\end{align}
We see that the period lattice of $f_2$ is rectangular, $v_2(L(f_2, \chi_4, 1)/\Omega_{f_2}^-)=0$ and the global minimal Weierstrass model of $E_{f_2}$ is represented as $y^2+y=x^3-x$ with Cremona label \texttt{37a1}. In \textsc{Table} \ref{tab:table2}, we give the comparison between the 2-adic valuation of the algebraic part $v_2(L(f_2, \chi_{4m}, 1)/\Omega_{f_2}^{\sgn(\chi_{4m})})$ and the lower bound in Theorem \ref{thm:MainThm3} for various integers $m$ with $r(m)=1$.

\newpage 

\begin{table}[ht]
    \caption{Comparison between the 2-adic valuation of the algebraic part and the lower bound in Theorem \ref{thm:MainThm1} for $N=34$ and $r(m)=1$.}
    \label{tab:table1}
    \centering
    \begin{tabular}{c|c|c|c|c}
        $m=q$ & $\sgn(\chi_{q})$ & $v_2(a_q-2)$ & $v_2\skakko*{L(f_1, \chi_m, 1)/\Omega_{f_1}^{\sgn(\chi_m)}}$ & Lower bound \\ \hline
        \red{5} & \red{$+1$} & \red{1} & \red{1} & \red{1} \\
        7 & $-1$ & 1 & 1 & 1 \\
        11 & $-1$ & 2 & $+\infty$ & 2 \\
        19 & $-1$ & 1 & 2 & 1 \\
        23 & $-1$ & 1 & 1 & 1 \\
        \red{29} & \red{$+1$} & \red{1} & \red{1} & \red{1} \\
        31 & $-1$ & 1 & 1 & 1 \\
        \red{37} & \red{$+1$} & \red{1} & \red{1} & \red{1} \\
        41 & $+1$ & 2 & $+\infty$ & 2 \\
        43 & $-1$ & 1 & 4 & 1 \\
        47 & $-1$ & 1 & $+\infty$ & 1 \\
        59 & $-1$ & 1 & $+\infty$ & 1 \\
        \red{61} & \red{$+1$} & \red{1} & \red{1} & \red{1} \\
        67 & $-1$ & 1 & 2 & 1 \\
        71 & $-1$ & 1 & 1 & 1 \\
        79 & $-1$ & 1 & 1 & 1 \\
        83 & $-1$ & 1 & $+\infty$ & 1 \\
        97 & $+1$ & 2 & $+\infty$ & 2 \\
        103 & $-1$ & 1 & $+\infty$ & 1 \\
        \red{109} & \red{$+1$} & \red{1} & \red{1} & \red{1} \\
        127 & $-1$ & 1 & $+\infty$ & 1 \\
        137 & $+1$ & 2 & 2 & 2 \\
        149 & $+1$ & 2 & $+\infty$ & 2 \\
        151 & $-1$ & 1 & $+\infty$ & 1 \\
        157 & $+1$ & 2 & $+\infty$ & 2 \\
        167 & $-1$ & 1 & 1 & 1 \\
        \red{173} & \red{$+1$} & \red{1} & \red{1} & \red{1} \\
        179 & $-1$ & 1 & 2 & 1 \\
        \red{181} & \red{$+1$} & \red{1} & \red{1} & \red{1} \\
        191 & $-1$ & 1 & $+\infty$ & 1 \\
    \end{tabular}
\end{table}

\newpage

\begin{table}[ht]
    \caption{Comparison between the 2-adic valuation of the algebraic part and the lower bound in Theorem \ref{thm:MainThm2} for $N=37$ and $r(m)=1$.}
    \label{tab:table2}
    \centering
    \begin{tabular}{c|c|c|c|c}
        $m=q$ & $\sgn(\chi_{q})$ & $v_2(a_{q}-2)$ & $v_2\skakko*{L(f_2, \chi_{4m}, 1)/\Omega_f^{\sgn(\chi_{4m})}}$ & Lower bound \\ \hline
        3 & $-1$ & 0 & $+\infty$ & 0 \\
        5 & $+1$ & 2 & $+\infty$ & 2 \\
        7 & $-1$ & 0 & $+\infty$ & 0 \\
        11 & $-1$ & 0 & $+\infty$ & 0 \\
        13 & $+1$ & 2 & $+\infty$ & 2 \\
        17 & $+1$ & 1 & $+\infty$ & 1 \\
        19 & $-1$ & 1 & 1 & 1 \\
        29 & $+1$ & 2 & $+\infty$ & 2 \\
        31 & $-1$ & 1 & $+\infty$ & 1 \\
        \red{41} & \red{$+1$} & \red{0} & \red{0} & \red{0} \\
        47 & $-1$ & 0 & $+\infty$ & 0 \\
        \red{53} & \red{$+1$} & \red{0} & \red{0} & \red{0} \\
        59 & $-1$ & 1 & 3 & 1 \\
        61 & $+1$ & 1 & $+\infty$ & 1 \\
        67 & $-1$ & 1 & $+\infty$ & 1 \\
        71 & $-1$ & 0 & $+\infty$ & 0 \\
        \red{73} & \red{$+1$} & \red{0} & \red{0} & \red{0} \\
        79 & $-1$ & 1 & 1 & 1 \\
        83 & $-1$ & 0 & $+\infty$ & 0 \\
        89 & $+1$ & 1 & $+\infty$ & 1 \\
        97 & $+1$ & 1 & $+\infty$ & 1 \\
        \red{101} & \red{$+1$} & \red{0} & \red{0} & \red{0} \\
        107 & $-1$ & 1 & $+\infty$ & 1 \\
        109 & $+1$ & 1 & $+\infty$ & 1 \\
        113 & $+1$ & 2 & $+\infty$ & 2 \\
        127 & $-1$ & 0 & $+\infty$ & 0 \\
        131 & $-1$ & 1 & 1 & 1 \\
        139 & $-1$ & 1 & $+\infty$ & 1 \\
        \red{149} & \red{$+1$} & \red{0} & \red{0} & \red{0} \\
        151 & $-1$ & 1 & $+\infty$ & 1 \\
    \end{tabular}
\end{table}

\newpage

\bibliographystyle{abbrv}
\bibliography{references.bib}
%
%	bibtexによる参考文献の表示.
%	bibs.bibというbibファイルを作り, そこに文献データを登録しておく.
%	引用文献を表示させるには
%	「LaTeX」で1回タイプセット -> 「BibTeX」で1回タイプセット -> 「LaTeX」で3回タイプセット
%	とする. それ以降はbibファイルに加除がないなら「BibTeX」によるタイプセットは必要なし.
%	plain(jplain)は参考文献をアルファベット順で出力. unsrt(junsrt)は引用された順で出力する.
%	ただしlatexmkを使用したタイプセットならば上で述べた複数回のタイプセットは自動的に一度でやってくれる.

\end{document}